\documentclass[11pt,a4paper,oneside]{amsart}

\usepackage{amsmath}
\usepackage{amsfonts}
\usepackage{amssymb}
\usepackage{amsthm}
\usepackage{bbm}
\usepackage[utf8x]{inputenc}
\usepackage[english]{babel}
\usepackage{enumerate}
\usepackage{graphicx} 
\usepackage{tikz}
\usepackage{url}

\usepackage[boxed,vlined,norelsize,oldcommands]{algorithm2e}
\SetKwInOut{Input}{input}
\SetKwInOut{Output}{output}
\SetArgSty{rm}
\SetFuncSty{tt}
\newcommand{\R}{{\mathbb R}}
\newcommand{\Z}{{\mathbb Z}}
\newcommand{\RR}{{\mathbb R}}
\newcommand{\ZZ}{{\mathbb Z}}
\newcommand{\zz}{{\ZZ}}
\newcommand{\rr}{{\RR}}
\newcommand{\nn}{{\mathbb N}}

\newcommand{\vones}{\mathbbm{1}}

\DeclareMathOperator{\conv}{conv}

\newcommand\core{\operatorname{core}}
\newcommand\Rep{\operatorname{Rep}}

\newcommand\GL{{\operatorname{GL}}} % general linear group; usage: \GL_n\RR
\newcommand\Fix{\operatorname{Fix}} % fixed points
\newcommand\ILP{\operatorname{ILP}} % ILP; usage: \ILP(A,b,c)
\newcommand{\Altern}[1]{\mathcal{A}_{#1}}
\newcommand{\Symmet}[1]{\mathcal{S}_{#1}}
\newcommand\convSet{S} % convex set in \RR^n
\theoremstyle{plain}
\newtheorem{theorem}{Theorem}
\newtheorem{proposition}[theorem]{Proposition}
\newtheorem{corollary}[theorem]{Corollary}

\theoremstyle{definition}
\newtheorem{definition}[theorem]{Definition}

\newtheorem{example}[theorem]{Example}

\newtheorem{remark}[theorem]{Remark}

\newcommand{\hk}[1]{H_{\vones,#1}}

\newcommand{\barymap}{\Phi}

\newcommand\id{\operatorname{id}}

\newcommand{\bigSetOf}[2]{\left\{#1\,:\,#2\right\}}

\title[Exploiting Symmetry in Integer Convex Optimization]{Exploiting Symmetry in Integer Convex Optimization using Core Points}
\author[K. Herr, T. Rehn, and A. Sch\"urmann]{Katrin Herr, Thomas Rehn, and Achill Sch\"urmann}


\address{%
Technische Universität Darmstadt,
Fachbereich Mathematik, %AG Optimierung,
Dolivostraße 15,
64293 Darmstadt, 
Germany}
\email{herr@mathematik.tu-darmstadt.de}

\address{% 
Institute of Mathematics,
University of Rostock,
18051 Rostock,
Germany}
\email{\{achill.schuermann,thomas.rehn\}@uni-rostock.de}

\begin{document}

\begin{abstract}We consider convex programming problems with integrality constraints
that are invariant under a linear symmetry group. To decompose such problems we introduce the new concept of core points, i.e.,
integral points whose orbit polytopes are lattice-free.
For symmetric integer linear programs we describe two algorithms based on this
decomposition. Using a characterization of core points for direct products
of symmetric groups, we show that prototype implementations can compete
with state-of-the-art commercial solvers, and solve an open MIPLIB problem. 
\end{abstract}

\maketitle
%\tableofcontents

\section{Introduction}

Symmetry is a fascinating and central subject in science.
It occurs frequently in the formulation of optimization problems.
In convex, and particularly in linear programming the size of a 
problem formulation can be decreased if the problem formulation
is invariant under a linear symmetry group.
In that case, the optimization problem
can simply be restricted to the linear subspace that is pointwise preserved 
by the symmetry group (see Section~\ref{sec:symconvsets}).

If some of the variables are constrained to integers, the situation becomes much more difficult.
In fact, symmetries are often considered rather problematic in that case, as
standard methods like {\em branch-and-bound}
or {\em branch-and-cut} (see \cite{schrijver-1986} and \cite{wolsey-1998})
have to solve a lot of equivalent sub-problems in such cases.
%Therefore, symmetric optimization problems are often considered as difficult test cases.
%Nevertheless, in recent years several researchers have discovered
Nevertheless, in recent years methods for exploiting symmetries in integer linear programming have been developed;
see for example
\cite{margot-2003}, \cite{friedman-2007}, \cite{bm-2008},
\cite{kp-2008}, \cite{olrs-2008}, \cite{lmt-2009}.
These specific methods fall into two main classes:
They either modify the standard branching approach,
using isomorphism tests or isomorphism free generation to
avoid solving equivalent subproblems; or they use techniques to
cut down the original symmetric problem to a lesser symmetric one,
which contains at least one element of each orbit of solutions.
For further reading we refer to the excellent survey~\cite{margot-2009}.
By now, two of the
leading commercial solvers, \cite{cplex} and \cite{gurobi}, 
have included some techniques
to detect and use available symmetry.

In contrast to the aforementioned approaches we use the rich geometric
properties of the involved symmetric convex sets (polyhedra).
Each symmetric convex set decomposes into a
lower dimensional part that is fixed pointwise by the symmetry group 
and symmetric slices orthogonal to it. 
We show that these slices contain integral points if and only if they contain certain integral {\em core points}.
A core point is an integral point
for which the convex hull of its orbit does not contain integral points other than the orbit points themselves
(see Definition~\ref{def:corepoint}).
In the case of a one dimensional pointwise fixed subspace,
and with the full symmetric group~$S_n$ %or the alternating group $A_n$ 
acting transitively on the coordinates of an integer linear program in $\R^n$,
these core points have been used in~\cite{bhj-2011}, to obtain a fast \emph{Core Point Algorithm}.
%running linear in the number of constraints and quadratic in the dimension. 
In this paper we generalize their concept of core points, 
for the design of algorithms that work for more general symmetry groups.
Based on core points, 
we describe two straightforward algorithms for solving integer linear programs.
A first approach can 
outperform state-of-the art commercial solvers on highly symmetric problems.
With a second approach we solve a benchmark 
problem from~\cite{miplib} that is currently marked ``open''.

The paper is organized as follows.
In Section~\ref{sec:symconvsets} we introduce 
some notation and recall essential facts about linear symmetry groups,
convex sets and their interplay.
In Section~\ref{sec:coresets} we introduce core points and
collect some basic results.
In Section~\ref{sec:characterization},
we give a complete characterization of core points with respect to direct
products of symmetric groups. 
In this case, all core points are found to be ``near'' the
pointwise fixed subspace. 
We show that this ideal situation cannot be expected for all groups, 
by describing arbitrarily ``far off'' core points for each cyclic group of even order. 
In Section~\ref{sec:sym-fibs} and Section~\ref{sec:param} we present two simple approaches for solving symmetric integer linear programs that use the concept of core points.
The first approach is based on enumeration of lattice points 
in the pointwise fixed subspace.
%and in a sense generalizes the \emph{Core Point Algorithm} in~\cite{bhj-2011}.
The second approach uses a parametrization of core points.
In Section~\ref{sec:experiments} we provide some information on computational experiments
that were conducted with prototype implementations of the proposed algorithms.
In Section~\ref{sec:discussion} we discuss some open problems
and give an outlook on potential future developments, 
based on the idea of core points.

\section{Linear symmetries of convex sets}\label{sec:symconvsets}

There is a canonical approach for handling symmetries
in symmetric convex optimization problems 
without integer restrictions.
% for instance, linear and semidefinite programming,  
Given a finite linear group $\Gamma\leq\GL_n(\RR)$, we consider the \emph{fixed~space}
\[
\Fix_\Gamma(\RR^n) = \bigSetOf{x\in\RR^n}{\gamma x=x\text{ for all }\gamma\in\Gamma}.
\] 
Note that the fixed space can be computed as the intersection %$\bigcap_{i} \Eig(\gamma_i, 1)$ 
of all group generator eigenspaces with eigenvalue~$1$. %$\{\gamma_1, \dots, \gamma_s\} \subseteq \Gamma$.
Let
$\barymap_\Gamma:\ \RR^n\to\Fix_\Gamma(\RR^n)$, $v\mapsto\beta(\Gamma v)$
be the orthogonal linear projection,
where $\beta(\Gamma v) := \frac{1}{\lvert \Gamma\rvert}\sum_{\gamma \in \Gamma}
\gamma v$ is the barycenter of the $\Gamma$-orbit of $v$. 
With respect to $\Gamma$, the space $\R^n$ decomposes into the
pointwise fixed $\Fix_\Gamma(\RR^n)$ and symmetric 
{\em fibers} (preimages of $\barymap_\Gamma$) orthogonal to it.
This {\em fibration} of $\RR^n$ carries over to symmetric convex sets.

A subset $\convSet$ of $\RR^n$ is \emph{$\Gamma$-symmetric} if
$\gamma\convSet=\convSet$ for every $\gamma\in\Gamma$. 
For every element $v$ of a $\Gamma$-symmetric set~$\convSet$, the
full orbit $\Gamma v$ is also contained in~$\convSet$.
If~$\convSet$ is convex, then for all $v\in\convSet$
also the orbit barycenter $\beta(\Gamma v)$ is contained in $\convSet$.
This implies that the intersection $\convSet\cap\Fix_\Gamma(\RR^n)$ for a convex set $\convSet$ is equal to the projection $\barymap_\Gamma(\convSet)$.
Thus we have the following proposition, which is an essential ingredient of symmetry exploiting techniques for convex optimitization problems;
see for instance~\cite{bhj-2011}, \cite{gp-2004}, \cite{bgsv-2012} or \cite{liberti-2012}.

\begin{proposition}\label{prop:reduce_conv_set_to_fix}
Let $\Gamma\leq\GL_n(\RR)$. Then
a $\Gamma$-symmetric convex set $\convSet\subseteq\RR^n$ is nonempty (feasible) if and only if the intersection $\convSet\cap\Fix_\Gamma(\RR^n)$ is nonempty (feasible).
\end{proposition}

\section{Core sets}\label{sec:coresets}

We now turn to the case where variables are integers.
Here, not all linear symmetry groups occur, as many of them do not preserve the integral lattice. As usually done in integer linear programming, we restrict ourselves to the case of 
symmetry groups that are subgroups of the full symmetric group $\Symmet{n}$, acting
on $\RR^n$ via permuting the standard basis vectors.
%All results presented below also apply to problems with symmetry groups in
%$\Orth_n\ZZ$ that are conjugate to subgroups of $\Symmet{n}$ in $\Orth_n\ZZ$.
%%by considering the conjugate problem.
%In such a case conjugation simply corresponds to a replacement of some variables $x_i$ 
%with their negative $-x_i$ to ensure that the symmetry group is a permutation of the standard basis.
We like to note, that throughout this paper we always work with linear representations of groups. 
%The isomorphism type of an abstract group does not contain enough information for our purposes. 
If we consider a subgroup of the symmetric group $\Symmet{n}$, we always implicitly assume that it is given by a standard representation with permutation matrices of size $n\times n$.

A statement similar to Proposition~\ref{prop:reduce_conv_set_to_fix} does not hold in general for problems with integrality constraints.
The barycenter of the orbit of an integer point does not need to be integral. 
With the following notion, however, we can obtain a statement similar to
Proposition~\ref{prop:reduce_conv_set_to_fix}.

\begin{definition} \label{def:corepoint}
  Given a group $\Gamma\leq\Symmet{n}$, a \emph{core point with respect to $\Gamma$} is an integral point
  $z\in\ZZ^n$ such that the convex hull of its $\Gamma$-orbit does not contain any further
  integral points, that is, $\conv(\Gamma z)\cap \ZZ^n=\Gamma z$. 
  
  A \emph{core set with respect to $\Gamma$} of a $\Gamma$-symmetric convex set $S$ is the set of all core points contained in $S$. We denote this core set by $\core_\Gamma(S)$. 
\end{definition}

Note that our definition of core points generalizes the notion used in~\cite{bhj-2011}.
They consider the case of $\Gamma = \Symmet{n}$ or $\Gamma = \Altern{n}$ with core points being defined as the integral points closest to the one-dimensional fixed space.

In the literature the polytope $\conv(\Gamma x)$ is known as \emph{orbit
  polytope} (see for example~\cite{MR1239510}). 
To the best of our knowledge, the lattice-freeness of such objects has not been studied before. 
\begin{remark}\label{rem:orbit-polytope-sphere}
Since $\Gamma$ is an orthogonal group,
all vertices of the orbit polytope $\conv(\Gamma z)$ with $\Gamma \leq \Symmet{n}$ lie on a Euclidean sphere centered at the orbit barycenter $\beta(\Gamma z)$ in the fixed space.
\end{remark}
Using this remark, we obtain the following discrete version of Proposition~\ref{prop:reduce_conv_set_to_fix}.
\begin{theorem}\label{theo:reduce_conv_set_Z_to_core_set}  
  Let $\Gamma\leq\Symmet{n}$ and $S \subset \rr^n$ be a $\Gamma$-symmetric convex set.
  Then $S$ contains integer points if and only if the intersection $S\cap\core_\Gamma(\rr^n)$ is nonempty.
\end{theorem}
\begin{proof}
  Let $z$ be an integer point of $S$.
  Since $S$ is invariant under $\Gamma$, the set of all orbit points $\Gamma z$ is also contained in $S$. 
  By convexity of $\convSet$, the convex hull $\conv(\Gamma z)$ is contained in $\convSet$ as well.
  We observe that $\core_\Gamma(C_1) \subseteq \core_\Gamma(C_2)$ for any two $\Gamma$-symmetric convex sets $C_1 \subseteq C_2$.
  Thus $\core_\Gamma(\conv(\Gamma z)) \subseteq \core_\Gamma(S) = S \cap \core_\Gamma(\rr^n)$.
  Since $\conv(\Gamma z)$ contains an integer point, it must also contain an integer point~$y$ with minimal Euclidean norm.
  By Remark~\ref{rem:orbit-polytope-sphere} the only integral points in the orbit polytope $\conv(\Gamma y)$ can be its vertices.
  Thus $y$ is a core point and is contained in the core set of $\conv(\Gamma z)$.
 The claim of the theorem follows.
\end{proof}

Note that, by definition, the core set $\core_\Gamma(S)$ of any $\Gamma$-symmetric set $S$ is also $\Gamma$-symmetric.
This important property implies that it is enough to consider only one arbitrary representative for every $\Gamma$-orbit of $\core_\Gamma(S)$.
We denote by $\Rep(\convSet)$ a set of representatives of a symmetric set $\convSet$.
Thus we have the following corollary, which simplifies the application of Theorem~\ref{theo:reduce_conv_set_Z_to_core_set} in practice.
\begin{corollary}\label{cor:ilp-feasibility-coreset}
 Let $\Gamma\leq\Symmet{n}$ and $S \subset \rr^n$ be a $\Gamma$-symmetric convex set.
 Then $S$ contains integer points if and only if the intersection $S\cap\Rep(\core_\Gamma(\rr^n))$ is nonempty.
\end{corollary}

By definition, the core set $\core_\Gamma(\convSet)$ can be obtained as the intersection~$\core_\Gamma(\rr^n) \cap \convSet$.
Therefore, the computation of $\core_\Gamma(\rr^n)$, which depends only on $\Gamma$, can be done in a separate step.
This allows the reuse of the result for problems that share the same linear representation of~$\Gamma$ as a symmetry group.
In practice it suffices to keep a list of representatives because we similarly obtain $\Rep(\core_\Gamma(\convSet))$ for every $\Gamma$-symmetric convex set $\convSet$ as the intersection~$\Rep(\core_\Gamma(\rr^n))\cap \convSet$.

Note that for binary variables the search space is reduced by the same amount as by previously known techniques like adding symmetry-breaking inequalities or isomorphism pruning.
For general integer variables, however, the reduction to core sets goes beyond pruning isomorphic nodes from the branch-and-bound tree.

We conclude this section with some elementary properties of core sets.
\begin{remark}\label{rem:core-sets-properties}
Given a group $\Gamma\leq\Symmet{n}$ and a $\Gamma$-symmetric convex set $\convSet\subseteq\RR^n$, it holds:
\begin{itemize}
 \item $\core_\Gamma(\convSet) \subseteq \core_{\Gamma'}(\convSet)$ for every subgroup $\Gamma'\leq\Gamma$. \\
  In particular, $\core_{\Symmet{n}}(\convSet) \subseteq \core_\Gamma(\convSet)$ for all $\Gamma \leq \Symmet{n}$.
 \item $\core_\Gamma(\convSet + z) = \core_\Gamma(\convSet) + z$ for every $z \in \Fix_{\Gamma}(\rr^n) \cap \zz^n$
 \item $\core_\Gamma(S) = -\core_\Gamma(-S)$
\end{itemize}
\end{remark}

\section{Characterization of core sets}\label{sec:characterization}

For the full symmetric group $\Symmet{n}$ and the alternating group $\Altern{n}$ on $n$ variables a complete characterization of $\Rep(\core_\Gamma(\rr^n))$ directly follows from~\cite[Theorem~3]{bhj-2011}:

\begin{remark}\label{remark:characterization_of_core_sets_for_An_Sn} 
  Let $H_{\vones, k}:=\bigSetOf{x\in\RR^n}{\langle x,\vones\rangle=k}$, where
  $0\leq k < n$, and $\vones$ is the all ones vector.
  Let $e_i$ denote the $i$-th standard basis vector in
  $\RR^n$. Given $\Gamma=\Symmet{n}$ or $\Gamma=\Altern{n}$, the core set
  $\core_\Gamma(H_{\vones, k})$ is the set $H_{\vones, k} \cap \{0,1\}^n$ of all $0/1$-vectors with $k$ ones, 
  i.e. vertices of a hypersimplex.
  The set of representatives for the complete core set is given by 
\[\Rep(\core_\Gamma(\rr^n))=\bigSetOf{ t\cdot\vones + \sum_{i=1}^ke_i}{0\leq
  k < n,t\in\ZZ}\]
  because the core set of every hyperplane $H_{\vones, k}$ can be represented by only one vector
  % by the $n$-fold and $(n-2)$-fold transitivity of $\Symmet{n}$ and $\Altern{n}$.
\end{remark}

\subsection{Core sets for direct products of groups}\label{subsec:core_sets_for_dir_prod}

A natural question is how the group structure is reflected in its core sets.
In this section we give an answer for direct products of groups.

Let $\Gamma=\Gamma_1\times\dots\times\Gamma_m$ be a subgroup of $\Symmet{n}$.
We assume that each $\Gamma_i$ acts on a coordinate subspace $X_i$ of $\RR^n$
and fixes all other coordinates. Thus we have a decomposition of $\RR^n$ into a
Cartesian product $X_1 \times \cdots \times X_m$. The subspaces~$X_i$ are pairwise disjoint by the assumed product structure of $\Gamma$.
We call the $X_i$ \emph{canonically associated} to $\Gamma$.

Convexity and the concept of direct products go well together, as
$\conv(X)=\conv(X_1)\times\dots\times\conv(X_m)$. This property carries over
to core sets.

\begin{theorem}\label{theo:direct_prod_of_core_sets}
 Let $\Gamma=\Gamma_1\times\dots\times\Gamma_m$ be a subgroup of $\Symmet{n}$ 
and $X_1, \dots, X_m$ the canonically associated coordinate subspaces of $\RR^n$.
Then $\core_\Gamma(Y_1 \times \cdots \times Y_m) =
  \core_{\Gamma_1}(Y_1) \times \cdots \times \core_{\Gamma_m}(Y_m)$ for
  arbitrary $\Gamma_i$-symmetric, convex subsets $Y_i$ of $X_i$.
\end{theorem}
\begin{proof}
  Let $Y:=Y_1\times\dots\times Y_m$. The orbit of $Y$ under $\Gamma$
  is the direct product of the orbits $\Gamma_iY_i$, thus
  $\conv(\Gamma
  Y)=\conv(\Gamma_1Y_1)\times\dots\times\conv(\Gamma_mY_m)$.
  Now, consider an integer
  point $y\in Y$ with a component $y_i\in Y_i$ which is not a core
  point with respect to $\Gamma_i$. Then there exists an integer point
  in $Y_i$, which is a strict convex combination of orbit points
  $\Gamma_i y_i$. It induces an (integral) convex combination of points
  in $\Gamma y$. Thus $y$ is not a core point with respect to
  $\Gamma$. The same argument shows that a point $y\in Y$ which is not
  a core point, always contains a component $y_i\in Y_i$ that is not a
  core point with respect to $\Gamma_i$. This completes the proof.
\end{proof}
Hence, for groups with known core sets
Theorem~\ref{theo:direct_prod_of_core_sets} yields a concrete
characterization of the core set of the direct product of such
groups. For the following result about products of symmetric and
alternating groups we use the characterization given in
Remark~\ref{remark:characterization_of_core_sets_for_An_Sn}.

\begin{corollary}\label{cor:characterization_of_core_sets_for_direct_prod_of_An_Sn}
Let $\Gamma:=\Gamma_1\times\dots\times\Gamma_d$ be a subgroup of $\Symmet{n}$ with either $\Gamma_i=\Symmet{k_i}$ or $\Gamma_i=\Altern{k_i}$ for $1\leq i\leq d$.
$\RR^n$ is decomposed into canonically $\Gamma$-associated disjoint subspaces~$X_i$ of dimension $k_i$. Then
  \begin{equation}\label{eq:core-sets-An-Sn}
\bigSetOf{(t_1\vones +\sum_{j=1}^{r_1}e_j,\dots,t_d\vones+\sum_{j=1}^{r_d}e_j)}{1\leq i\leq d,\;0\leq r_i < k_i,\; t_i\in\ZZ }
  \end{equation}
describes a set of representatives for $\core_\Gamma(\rr^n)$.
Here, we denote the elements of $\RR^n=X_1\times \dots \times X_d$ by tuples $(x_1,\dots, x_d)$, with $x_i\in X_i$.
\end{corollary} 
Note that
Corollary~\ref{cor:characterization_of_core_sets_for_direct_prod_of_An_Sn}
also applies to situations where symmetric or alternating groups act only
on some subspaces of $\RR^n$ as we can split the remaining
subspace into a direct product of one-dimensional subspaces, with an
action of $\Symmet{1}$ on each of them. 
In the corollary, $d$ is the dimension of the fixed space~$\Fix_\Gamma(\RR^n)$.
Note that in Theorem~\ref{theo:direct_prod_of_core_sets}, the number $m$ of factors is equal to the dimension $d$ of the fixed space if and only if all groups $\Gamma_i$ are transitive (see the following Section~\ref{sec:transitive-groups}).

In Sections~\ref{sec:sym-fibs}~and~\ref{sec:param} we will make use of the
core set characterization given in Corollary~\ref{cor:characterization_of_core_sets_for_direct_prod_of_An_Sn} in two different approaches for solving
integer linear programs whose symmetry groups are of this special
form.

\subsection{Transitive group actions}\label{sec:transitive-groups}
Remark~\ref{rem:core-sets-properties} states that the core set with respect to a group is contained in the core set with respect to each of its subgroups. 
Every group is a subgroup of the direct product of transitive groups (cf. \cite[Sec~5.5]{MR0103215}). 
%ALSO: Every group is the subdirect product of transitive groups. 
Groups acting on the standard basis of~$\RR^n$ are transitive if all basis vectors are in the same orbit.
In this section we focus on these basic building blocks, that is, on permutation groups which act transitively.

For a transitive group $\Gamma$ the fixed space is the linear span of the all-ones vector $\vones$.
For every integral point $z\in\zz^n$ there is a unique $k = \langle z, \vones \rangle \in \zz$ such that $z$ is contained in the hyperplane $\hk{k} = \bigSetOf{x\in\RR^n}{\langle x,\vones\rangle=k}$.
Thus $\core_\Gamma(\rr^n)$ is the infinite union of the core sets of $\hk{k}$ for all $k \in \zz$.
Hence, it is enough to study core sets of the hyperplanes~$\hk{k}$.
We know from Remark~\ref{rem:core-sets-properties} that core sets are translation invariant under some conditions.
Especially, for a transitive group $\Gamma\leq\Symmet{n}$ we conclude that $\core_\Gamma(\hk{k+n}) = \core_\Gamma(\hk{k} + \vones) = \core_\Gamma(\hk{k}) + \vones$ for every hyperplane $\hk{k}$.
Thus it suffices to study core sets of hyperplanes with $k \in \{0, 1, \dots, n-1\}$.
Since $\core_\Gamma(\hk{0}) = \{0\}$ and $\core_\Gamma(\hk{n-k}) = - \core_\Gamma(\hk{k}) + \vones$, it is in fact enough to understand the core sets of $\hk{k}$ for $k \in \{1, \dots, \left\lfloor \frac{n}{2} \right\rfloor\}$.
As we can deduce all core sets from these,
we say that $\Gamma$ has a finite core set if $$\bigcup_{k=1}^{\left\lfloor \frac{n}{2} \right\rfloor} \core_\Gamma(\hk{k})$$ is a finite set.

From Remark~\ref{remark:characterization_of_core_sets_for_An_Sn} we know that the
symmetric and alternating group have finite core sets because $\hk{k}$ contains ${n
  \choose k}$ core points that all lie in the same orbit. 
%%More general, it can be shown that finite subgroups of $S_n$, which act irreducibly on $\R^n$, have finite core sets.

The following example shows that not all groups have a finite core set.
In particular, cyclic groups give rise to symmetric lattice-free simplices.
These simplices have asymptotically unbounded diameter and volume and can be thought of as symmetric siblings of Reeve's famous lattice-free simplices (cf.~\cite{MR0095452}).
More precisely, the following example shows that the core set of all cyclic groups of even order consists of infinitely many simplices and is thus infinite.

\begin{example}[Infinite Core Sets]\label{ex:cyclic-group-even-infinite-core-set}
 Let $\Gamma = \langle \sigma \rangle$ be a cyclic group of order $n = 2m \geq 4$ where $\sigma = (n\,\dots\,3\,2\,1)$.
 We will show that for every choice of parameters $a_2, \dots, a_{m} \in \zz$ the orbit polytope $P = \conv(\Gamma z)$ with $$z = (1,a_2,a_3,\dots,a_{m},0,-a_2,-a_3,\dots,-a_{m})$$ has no integer points besides its vertices.
 
 To see this let $v = \sum_{i=1}^{n} \lambda_i \sigma^{i-1} (z) \in P \cap \zz^n$ be an integer vector in $P$, represented as a convex combination with $\lambda_i \in [0,1]$ and $\sum_{i=1}^n \lambda_i = 1$.
 For each of the $n$ coordinates of $v$, this gives an equation in $\lambda_1,\dots,\lambda_n$.
 Summing up the first and the $(m+1)$-st equation we obtain:
 \begin{equation}\label{eq:cyclic-group-conv-coordinate-sum}
  \begin{split}
   & \left \langle \sum_{i = 1}^{n} \lambda_i \sigma^{i-1} \left(z\right) , e_1 + e_{m+1}\right \rangle\\
  =& \left ( \lambda_1 + \sum_{i = 2}^{m} \lambda_i a_i + \sum_{i = 2}^{m} \lambda_{i+m} (-a_i) \right)\\
  &\quad + \left ( 
     \sum_{i = 2}^{m} \lambda_{i} (-a_i) + \lambda_{m+1} + \sum_{i = 2}^{m} \lambda_{i+m} a_i \right) \\
  =& \lambda_1 + \lambda_{m+1}
  \end{split}
 \end{equation}
 Since all coordinates of $v$ are integer, the expression in~\eqref{eq:cyclic-group-conv-coordinate-sum} must be integer as well.
 This implies that $\lambda_1 +  \lambda_{m+1}$ must be zero or one.
 The same can be said about the sum of $\lambda_j$ and $\lambda_{m+j}$ in general by adding the $j$-th and $(m+j)$-th equation.
 Because $v$ is a convex combination, there is exactly one index $j$ such that $\lambda_j + \lambda_{m+j} = 1$ and all other $\lambda_i$ are equal to~$0$.
 Thus it must hold that $\lambda_j \sigma^{j-1}(z) + \lambda_{m+j} \sigma^{m+j-1}(z) \in \zz^n$.
 In particular, $\langle \lambda_j \sigma^{j-1}(z) + \lambda_{m+j} \sigma^{m+j-1}(z), e_j \rangle = \lambda_j$ must be integer.
 Hence the convex combination in $v$ must be trivial and $v$ is always a vertex of $P$.
 Thus $P$ cannot have integer points besides its vertices.

 Since the orbit polytope $P$ is contained in the hyperplane $\hk{1}$, it has dimension $n-1$.
 The full-dimensional simplex $P'=\conv(\{0\} \cup P)$ is still symmetric and does not contain integer points except its vertices.
 Note that the volume of $P'$ can be computed as the absolute value of a determinant of a circulant matrix (cf.~\cite{MR543191}).
 By looking at the eigenvalues of this matrix, we conclude that for fixed $n$ these families of symmetric simplices parametrized by $a_2, \dots, a_m$ have unbounded volume.
%$\blacksquare$
\end{example}

\section{Symmetric fibrations of Integer Linear Programs}\label{sec:sym-fibs}

In this and the following section we turn to the application of core sets to integer linear programming. 
We consider integer linear programs (ILPs) of the following form:
\begin{equation*}
  \begin{array}{ll}
    \max & \langle c, x\rangle\\
    \textrm{s.t.} & Ax \le b \, , \ x \in \ZZ^n  \, ,
  \end{array}
\end{equation*}
with $A \in \rr^{m\times n}$, $b \in \rr^m$, and objective $c \in \rr^n$. 
To abbreviate such an instance we use the notation $\ILP(A,b,c)$. 
The corresponding polyhedron is denoted by~$P=\{ x \in \RR^n \, : \,  Ax \le b \}$. 
As previously mentioned, we only consider symmetry groups of integer linear programs of dimension $n$ that are subgroups of the symmetric group $\Symmet{n}$, acting on $\R^n$ by permuting coordinates.
\begin{definition}
  An integer linear program $\ILP(A,b,c)$ of dimension $n$ over a polyhedron $P$ is \emph{$\Gamma$-symmetric} with respect to a group $\Gamma\leq\Symmet{n}$ if the following two conditions holds:
  First, $Ax \leq b$ if and only if $A \gamma x\leq b $ for all $\gamma \in \Gamma$ and $x\in \zz^n$.
  Second, $\langle c, x\rangle = \langle c, \gamma x\rangle$  for all $\gamma \in \Gamma$ and $x\in P \cap \zz^n$.
  The group $\Gamma$ is then called a \emph{symmetry group} of $\ILP(A,b,c)$.
\end{definition}
Our first approach is a generalization of the \emph{Core Point Algorithm} by Bödi, Herr, and Joswig. 
In~\cite{bhj-2011}, the authors consider transitive group actions, which lead to a one-dimensional fixed space, the span of the all ones vector. 
They decompose the problem by intersecting the feasible region with all affine hyperplanes orthogonal to the fixed space containing integer points, and check the integer feasibility of these intersections. 
For the alternating group and the full symmetric group on all variables they show that it suffices to test one core point per intersection. 
In order to generalize this approach to fixed spaces of arbitrary dimensions we use a similar decomposition.

Given a symmetry group $\Gamma$ of an integer linear program, consider the projection $\Phi_\Gamma : \rr^n \to \Fix_\Gamma(\rr^n)$ onto the fixed space 
as introduced in Section~\ref{sec:symconvsets}. 
The projection $\barymap_\Gamma(\zz^n)$ of all integer vectors forms a lattice, which we denote by $\Lambda$.
For $x \in \Lambda$, that is, for a projected integer vector, we call the set of pre-images $\barymap_\Gamma^{-1}(x)$ a \emph{$\Gamma$-fiber}, or simply a \emph{fiber}.  
All feasible points of a $\Gamma$-symmetric integer linear program $\ILP(A,b,c)$ are contained in $\Gamma$-fibers that intersect the corresponding polyhedron~$P$.
\begin{remark}\label{rem:fiber-props}
The objective function $c$ is constant on each $\Gamma$-fiber.  
Moreover, $\Gamma$-fibers are $\Gamma$-invariant since $\barymap_\Gamma(x)=\barymap_\Gamma(y)$ for any $y\in\Gamma x$.  
The dimension of a $\Gamma$-fiber is equal to $\dim(\ker(\barymap_\Gamma))=n-\dim(\Fix_\Gamma(\RR^n))$.
\end{remark}
Thus, to solve an ILP it suffices to find a $\Gamma$-fiber with maximal objective value such that its intersection with $P$ contains an integer point. 
A naïve first approach is to simply enumerate all $\Gamma$-fibers and successively test their integer feasibility.
Algorithm~\ref{alg:sym-fib} summarizes this idea.
\begin{algorithm}[tb]
  \dontprintsemicolon
  \linesnumbered
  \KwIn{$\Gamma$-symmetric $\ILP(A,b,c)$ with bounded polyhedron $P$}
  \KwOut{optimal solution of $\ILP(A,b,c)$ or ``infeasible''}
  $c_{\text{opt}} \leftarrow -\infty$\;
  \ForEach{$z \in \barymap_\Gamma(P) \cap \Lambda$}{\label{algstep:sym-fib_enum}
  $F \leftarrow \barymap_\Gamma^{-1}(z)$\;
    \If{$\langle c,z\rangle > c_{\text{opt}}$ {\bf and} $F\cap P \cap \zz^n$ nonempty }{\label{algstep:sym-fib_feasibility}
      $z_{\text{opt}} \leftarrow \text{arbitrary element of } F\cap P \cap \zz^n$\;
      $c_{\text{opt}} \leftarrow \langle c,z\rangle$\;
    }
  }
  \Return{$z_{\text{opt}}$ or ``infeasible''}
  \caption{Naïve fiber enumeration of $\ILP(A,b,c)$}
  \label{alg:sym-fib}
\end{algorithm}
Note that we require boundedness of the polyhedron $P$ to guarantee termination of the algorithm. 
The two computational challenges are the enumeration of all lattice points in $\barymap_\Gamma(P) \cap \Lambda$ in Step~\ref{algstep:sym-fib_enum} and testing integer feasibility of~$F\cap P$ in Step~\ref{algstep:sym-fib_feasibility}. 

Because of the symmetry and since we project onto the fixed space, we can easily obtain an inequality (facet) description of the projected polyhedron $\barymap_\Gamma(P)$ by computing the barycenter of each orbit of facet normals of the original polytope (cf.~\cite[Theorem~1]{bhj-2011}).
The lattice $\Lambda=\barymap_\Gamma(\ZZ^n)$ in the fixed space turns out to be always an isometric embedding of an integer lattice~$\ZZ^d$ with scaled coordinate axes.
%(Scaling factors of coordinates are $1/\sqrt{k_i}$, for each block of $k_i$ coordinates that the group acts on transitively, $n=k_1+\ldots + k_d$).
Therefore, we have to deal with integer lattice point enumeration in dimension~$d$ of the fixed space.

As a refinement of the algorithm we could first solve an integer linear program in dimension~$d$ over $\barymap_\Gamma(P) \cap \Lambda$ and enumerate the lattice points in the intersection with the affine hyperplane $H$ of the resulting objective value. 
If the corresponding fibers do not contain feasible integer points for the original problem, we cut off the halfspace defined by $H$ and iterate the procedure.
However, in the worst case all lattice points in $\barymap_\Gamma(P) \cap \Lambda$ need to be enumerated. 
To address the second challenge, the integer feasibility check for the $\Gamma$-symmetric set~$F\cap P$ in Step~\ref{algstep:sym-fib_feasibility}, we use core sets. 
The following corollary is immediate from Theorem~\ref{theo:reduce_conv_set_Z_to_core_set}.
\begin{corollary}\label{cor:fiber-feasibility-coreset}
  Let $\ILP(A,b,c)$ be a $\Gamma$-symmetric integer linear program over the polyhedron~$P$. 
  Further, let $x\in\Lambda$ and $F = \barymap_\Gamma^{-1}(x)$. 
  Then there is an integral vector in $F\cap P$ if and only if the intersection $P\cap\Rep(\core_\Gamma(F))$ contains one.
\end{corollary}
Thus we can check the integer feasibility of the fiber section $F \cap P$ by testing at most $\lvert\Rep(\core_\Gamma(F))\rvert$ many points for containment in~$P$.
Under certain circumstances we know that the set $\Rep(\core_\Gamma(F))$ is small, as the following considerations show.

If the automorphism group $\Gamma$ is a direct product of groups $\Gamma_1, \dots, \Gamma_k$, then every $\Gamma$-fiber $F$ is the Cartesian product of $\Gamma_i$-fibers $F_i$.
Thus by Theorem~\ref{theo:direct_prod_of_core_sets} we know that $\core_\Gamma(F)$ is the product of the fiber core sets $\core_{\Gamma_i}(F_i)$.
In particular, the set of representatives $\Rep(\core_\Gamma(F))$ is a direct product of fiber core set representatives $\Rep(\core_{\Gamma_i}(F_i))$.
For a direct product of symmetric products, we therefore have the following theorem.
\begin{theorem}\label{thm:ilp-symmetric-product}
 Let $\Gamma = \Symmet{k_1} \times \cdots \times \Symmet{k_d}$ be a subgroup of an ILP over a polytope $P \subset \rr^n$ given by $m$ linear constraints.
 Assume that the set $L=\barymap_{\Gamma}(P)\cap \Lambda$ (with $\Lambda=\barymap_{\Gamma}(\Z^n)$) is given.
 Then there is an $O(\lvert L \rvert \cdot mn)$-time algorithm that either finds an optimal solution or detects integer infeasibility.
\end{theorem}
\begin{proof}
 In Algorithm~\ref{alg:sym-fib} we check for integer feasibility of $F \cap P$, for every fiber $F$ that intersects $P$.
 As every such fiber corresponds to an element of $L$, we have to perform at most $\lvert L\rvert$ many of these checks.
 By our previous considerations and Remark~\ref{remark:characterization_of_core_sets_for_An_Sn}, we know that $\Rep(\core_\Gamma(F))$ consists of only one element because $\lvert\Rep(\core_{\Gamma_i}(F_i))\rvert = 1$ for every fiber $F_i$ of $\Gamma_i = \Symmet{k_i}$.
 Thus every fiber feasibility check can be performed in $O(mn)$ time by Corollary~\ref{cor:fiber-feasibility-coreset}.
\end{proof}

Note that in Theorem~\ref{thm:ilp-symmetric-product} the difficulties are hidden in the computation of the set $L = \Phi_\Gamma(P)\cap\Lambda$ of all lattice points in the projected polytope.
Further research is needed to improve this rather naïve approach.
More elaborate methods could make use of the structure of the lattice $\Lambda$ and include more polyhedral information.

% \begin{algorithm}[tb]
%   \dontprintsemicolon
%   \linesnumbered
%   \KwIn{$\ILP(A,b,c)$ with symmetry group $\Gamma$}
%   \KwOut{optimal solution of $\ILP(A,b,c)$ or ``infeasible''}
%   $z_{\text{opt}} \leftarrow \emptyset$\;
%   $k \leftarrow \infty$\;
%   \While{$\max\limits_{z \in \{x \in \barymap_\Gamma(P) \cap \Lambda \;:\; c^tx \leq k\} } c^tz$ is feasible with solution $z_p$}{
%     \ForEach{$z \in \{x \in \barymap_\Gamma(P) \cap \Lambda \;:\; c^tx=c^tz_p\}$}{
%   $F \leftarrow \barymap_\Gamma^{-1}(z)$\;
%       \If{$F\cap P$ feasible}{
%         \Return{$z_p$}
%       }
%     }
%     $k \leftarrow c^tz_p - 1$\;
%   }
%   \Return{``infeasible''}
%   \caption{Enumerate fibers by objective value}
%   \label{alg:core-set-enumerate-layer-fiber}
% \end{algorithm}

\section{Core set parametrization}\label{sec:param}

In this section we explore a different approach to solve symmetric integer linear programs,
which is also based on core sets.

By Corollary~\ref{cor:ilp-feasibility-coreset} we know that a $\Gamma$-symmetric ILP over a polyhedron $P \subset \rr^n$ is feasible if and only if it contains an element of the core set $\core_\Gamma(\rr^n)$.
In particular, every feasible ILP has an optimal solution in $\core_\Gamma(\rr^n)$.
The symmetry of  the optimization problem even guarantees that we 
can always find an optimal solution in $\Rep(\core_\Gamma(\rr^n))$.
Thus we may solve the original ILP under the additional constraint that the solution is such a core set representative.
A natural approach to model this constraint would be to add linear inequalities. 
This would be similar to adding symmetry breaking constraints to ensure that solutions lie in a fundamental domain of the symmetry group (cf.~\cite{friedman-2007}, \cite{kp-2008}). 
In general, finding such inequalities is not easy and may even be practically impossible.

We propose a different way to model the constraint that the solution is a core set representative.
We introduce a parametrization of the core set, which is particularly interesting for symmetry groups with small core sets, for instance, in case of a direct product of symmetric groups.
We assume that $\Gamma$ contains a direct product $\Gamma' := \Symmet{k_1} \times \cdots \times \Symmet{k_d}$ of symmetric groups, where $n=\sum_{i=1}^d k_i$ and $d$ is the dimension of the fixed space $\Fix_{\Gamma'}(\rr^n)$.
Note that we can always find such a decomposition by choosing some of the $k_i$ to be~$1$.
From Corollary~\ref{cor:characterization_of_core_sets_for_direct_prod_of_An_Sn} we know a set of representatives for $\core_{\Gamma'}(\rr^n)$ and thus also for its subset $\core_{\Gamma}(\rr^n)$.
Hence, we have the following theorem.
\begin{proposition}\label{prop:parametrization-symmetric-product}
Every ILP with symmetry group $\Symmet{k_1} \times \cdots \times \Symmet{k_d}$ is feasible with optimal solution~$z$ if and only if there exist integers $t_i \in \zz$ and $s_{i,j} \in \{0,1\}$ for every $i \in \{1, \dots, d\}$ and $j \in \{1, \dots, k_i-1\}$ such that
\begin{equation}\label{eq:core-set-parametrization-An-Sn}
 z = \bigoplus_{i=1}^d \left( t_i \vones_{k_i} + \sum_{j=1}^{k_i-1} s_{i, j}\, c(j) \right),
\end{equation} 
under the constraint that $\sum_{j=1}^{k_i-1} s_{i,j} \leq 1$.
\end{proposition}

Note that \eqref{eq:core-set-parametrization-An-Sn} is just a reformulation of expression~\eqref{eq:core-sets-An-Sn} that is more convenient for the integer programming context.
In equation~\eqref{eq:core-set-parametrization-An-Sn} the term $c(j)$ denotes $\sum_{i=1}^j e_i$, which is a representative of the core set of $\hk{j}$.
We use the variables $s_{i,1}, \dots, s_{i, k_i-1}$ to select at most one vector $c(1), \dots, c(k_i-1)$.
By using the new variables $t_i$ and $s_{i,j}$ we are able to parametrize the set of core set representatives from Corollary~\ref{cor:characterization_of_core_sets_for_direct_prod_of_An_Sn}.

This means that we can replace the old variables $z_1, \dots, z_n$ by $n$ new variables $t_i$ and $s_{i,j}$.
By this transformation we reduce the search space (by reduction to core sets) and we eliminate all $\Gamma'$-symmetry (by reduction to core set representatives).
Moreover, we may remove inequality constraints during the execution of this transformation.
To illustrate the idea and to simplify notation, we assume that $\Gamma = \Symmet{n}$ is the full symmetric group.
Let
\begin{equation*}\label{eq:constraint-before-S_n-transformation}
 \langle a,z \rangle = \sum_{i=1}^n a_i z_i \leq b
\end{equation*}
be an inequality before the transformation.
Because $\Symmet{n}$ is an automorphism group of the problem, there is a permutation $\sigma$ such that 
\begin{equation}\label{eq:inequality-sort-descendingly}
a_{\sigma(1)} \geq a_{\sigma(2)} \geq \dots \geq a_{\sigma(n)} 
\end{equation}
and $\langle a^\sigma, z \rangle = \sum_{i=1}^n a_{\sigma(i)} z_i \leq b$
is a valid inequality.
Thus the condition $\langle a^\sigma, z(t,s_1, \dots, s_{n-1}) \rangle \leq b$ implies the constraint $\langle a, z(t,s_1, \dots, s_{n-1}) \rangle \leq b$ and all other constraints in its orbit.
Here we use the short notation $z(t,s_1, \dots, s_{n-1}) = t \vones_n + \sum_{j=1}^{n-1} s_j c(j)$ for the vector in~\eqref{eq:core-set-parametrization-An-Sn}.
As a byproduct we can eliminate all but one inequality from each orbit of inequalities by using the transformation.
By ordering coefficients block-wise descendingly as seen in \eqref{eq:inequality-sort-descendingly}, this elimination of constraints generalizes to the case of direct products of symmetric groups, which we discussed before.
As we have to touch every constraint anyway, this carries no significant overhead.

The parametrization idea and transformation induced by Proposition~\ref{prop:parametrization-symmetric-product} is not restricted to symmetric groups.
A similar statement is true for every group for which the complete core set or its representatives are known.
The simpler the core set can be parametrized, the fewer variables the transformation adds.
In the next section we will see how this approach works for symmetric groups in practice.

\section{Computational experiments}\label{sec:experiments}

To assess the practical feasibility of our proposed algorithms, we implemented prototypes.
First we describe our experiments with the approach discussed in Section~\ref{sec:sym-fibs}.

\subsection{Naïve fiber enumeration}\label{sec:experiments-fibers}

For Algorithm~\ref{alg:sym-fib} we use SCIP 2.0.1 \cite{scip} to enumerate all lattice points of the projected polyhedron.
Since we compare running times with commercial solvers which do not use exact arithmetic, this is a viable alternative to other lattice enumeration tools like \cite{latte} or \cite{normaliz}. 
Each enumerated point corresponds to a fiber. 
The integer feasibility of these fibers is tested by using core points.
Currently, our knowledge of core sets of groups beyond the alternating and the symmetric group and their direct products is limited.
Therefore we only implemented integer feasibility checks for these groups.
This core point check is realized in a dedicated program written in C++, which reads a polyhedron and a list of fibers.
It either returns an optimal fiber or reports that the input is infeasible.

Note that Algorithm~\ref{alg:sym-fib} is not always practically feasible.
Since we enumerate lattice points in the dimension~$d$ of the fixed space, the value of $d$ should not exceed about $10$ to remain tractable.
At the same time we have explicit complete core set descriptions for symmetric and alternating groups.
We therefore focus on problems with automorphism groups which are the product of ten or less symmetric groups.
Since we are not aware of problem instances in the literature meeting these conditions, we constructed problems ourselves.

We created random instances by the following scheme, using \cite{PermLib} and \cite{Polymake}.
For different values $d$ less than $10$ and different values of $k_1, \dots, k_d \in \nn_{\geq 1}$ we constructed ILPs in dimension $n = \sum_{i=1}^d k_i$ and with automorphism group $\Gamma = \Symmet{k_1}\times \cdots \times \Symmet{k_d}$.
We generated $3n$ inequalities $\langle a, x \rangle \leq b$ where 
$$a = \bigoplus_{i=1}^d f_i \left ( \sum_{j=1}^{k_i} a_{i,j} x_j \right ).$$
Here the $f_i$ were chosen independently uniformly at random from the set $\{1,\dots,20\}$.
The $a_{i,j}$ were zero with probability $0.1$ and otherwise selected uniformly at random from the set $\{5,\dots,15\}$.
The right hand side $b$ was set to $\lfloor 0.95 \cdot \langle a, \vones \rangle \rfloor$.
Finally, all inequalities in the orbit of $\Gamma$ were added and the domain of all variables set to $\zz_{\geq 0}$.
Additionally, to exclude the zero vector, we added the inequality $\sum_{i=1}^n x_i \geq 1$.
The objective function $c$ was chosen as $c = c_1 \vones_{k_1} \oplus \dots \oplus c_d \vones_{k_d}$ where the $c_i$ were chosen independently uniformly at random from the set $\{1,\dots,10\}$.

As the order of $\Gamma$ grows very quickly with $d$ and the $k_i$, we conducted our experiments only for selected values of these parameters.
For each $n \in \{10,12,15,18\}$ we tried to find three different groups $\Gamma$ each such that the number of constraints was comparable for different $n$.
We selected the parameters $d$ and $k_i$ so that we had one small instance and two large instances, the latter ones with different dimension $d$ of the fixed space.
The average ratio of non-zeros in the instances was about $90\%$, as was to be expected from the choice of random variables.

Table~\ref{tbl:random-experiments} shows the average results for $10$ randomly generated instances for every set of dimension parameters.
We performed the experiments on an Intel Core-i7 machine with eight logical CPUs at $2.8$ GHz and $16$ GB RAM.
We ran our tests with Gurobi 4.5.1 \cite{gurobi}, CPLEX 12.3 \cite{cplex} and our own fiber/core set-prototype.
We used the commercial solvers with their default settings and allowed $8$ threads.

\begin{table}
\begin{tabular}{lcclllll}
\hline
Groups & $n$ & $d$ & \#rows & Gurobi & CPLEX & Core\\
\hline
$(\Symmet{5})^2$ & 10 & 2 & 182151& 64.40 & 89.58 & 0.25\\
$\Symmet{5} \times \Symmet{3} \times \Symmet{2}$ & 10 & 3 & 23204  & 6.36 & 6.47 & 0.05\\
$\Symmet{8} \times (\id)^2 $ & 10 & 3 & 342289  & 170.39 & 1222.11 & 0.47\\
\hline
$(\Symmet{4})^3$ & 12 & 3 & 217273 & 123.77 & 129.34 & 0.36\\
$(\Symmet{3})^4$ & 12 & 4 & 28353 & 10.74 & 7.05 & 0.07\\
$\Symmet{6} \times \Symmet{4} \times (\id)^2$ & 12 & 4 & 236001  & 114.32 & 138.59 & 0.37\\
\hline
$(\Symmet{3})^5$ & 15 & 5 & 182366  &  136.90 & 107.94 & 0.41\\
$\Symmet{3} \times (\Symmet{2})^6$ & 15 & 7 & 11751 &  6.70 & 3.88 & 0.35\\
$(\Symmet{5})^2 \times (\id)^5$ & 15 & 7 & 267434  &  210.54 & 223.14 & 0.63\\
\hline
$(\Symmet{3})^4 \times (\Symmet{2})^3$ & 18 & 7 & 286732  &  304.81 & 278.72 & 1.48\\
$(\Symmet{2})^9$ & 18 & 9 & 18854 &  16.17 & 11.13 & 4.72\\
$\Symmet{5} \times \Symmet{3} \times (\Symmet{2})^4 \times (\id)^3$ & 18 & 9 & 315501 &  429.43 & 418.71 & 5.63\\
\hline
\end{tabular} 
\caption{Running times in seconds on random symmetrized instances, averaged on 10 runs each}
\label{tbl:random-experiments}
\end{table}

The results show that our code is faster than the commercial solvers on these instances.
We can also observe that the running time of our prototype increases significantly with $d$ because we have to enumerate lattice points in this dimension.
The input to our prototype included the symmetry group of the problem, so it did not have to be determined.
%As we do not have code that specifically detects symmetries which are a direct product of symmetric groups, we cannot tell how much of the running time advantage would be lost if the symmetry group were not known a-priori.

\subsection{Core set parametrization}

We also tested a transformation by Proposition~\ref{prop:parametrization-symmetric-product} on these instances.
This reduced the problems to instances with $60$ or less inequalities in dimensions $\{10,12,15,18\}$.
Since these are in general easy problems for ILP solvers,
we always have obtained the optimal solution in less than $0.1$ seconds regardless of the original problem size and the solver used.
As this transformation approach has no obvious limits on the problem size for which it is practically feasible, we were able to test it also on a real world problem.

Among all instances of the MIPLIB 2010 collection (cf.~\cite{miplib}) we looked for one which is small and whose symmetry group is large and consists to a large extent of a product of symmetric groups.
One of the candidates was {\tt toll-like}, a then open $0/1$-problem with $4408$ constraints in dimension $2883$.
Its symmetry group has $(\Symmet{2})^{230}$ as a subgroup.
After our transformation it had $4638$ constraints, still in dimension $2883$.
However, the presolvers of CPLEX, Gurobi and SCIP were able to eliminate $230$ variables, one for each $\Symmet{2}$ factor in the original problem.
Moreover, the number of constraints could be reduced to $3948$, which is $460$ less than in the original problem.
These reductions allowed us to solve this previously open problem with Gurobi 4.5.1 after about $4.5$ days on our workstation.
Under the same conditions solving the original, untransformed problem was not possible because both CPLEX and Gurobi ran into memory problems.

We thank the anonymous referee for a hint to the following analysis of the toll-like symmetry.
The instance contains 230 pairs of variables $x,y$ that appear in the following way:
\begin{equation}\label{eq:toll-like-inequalities}
\begin{split}
 x + R \geq 0,\qquad& y + R \geq 0,\\
 x - R \geq 0,\qquad& y - R \geq 0,
\end{split}
\end{equation}
where $R$ is some term in other variables.
In this case we can aggregate $x$ and $y$ into a single variable.
It is unclear why none of the tested solvers did perform this variable elimination that seems to be easily detectable without knowledge of core sets.
The transformation described in Section~\ref{sec:param} rewrites the contraints in such a way that solvers notice the redundancy in the model.

Since the problem contains only binary variables, the core set parametrization does not reduce the search space more than previously known techniques like symmetry-breaking inequalities.
Therefore it is an interesting question whether MIPLIB instances contain $\Symmet{k}$-symmetries on general integer variables that are not eliminated by current presolvers.
An ongoing collaboration with Marc Pfetsch~\cite{PfetschRehn2012} gave us access to the symmetry groups of most MIPLIB 2010 instances before and after presolving of SCIP and Gurobi.
All $\Symmet{k}$-symmetries on general integer variables in the original formulation are removed by the presolver of Gurobi.
However, presolving introduces new $\Symmet{2}$-symmetries on integer variables with non-binary bounds on five instances (\texttt{atlanta-ip}, \texttt{biella1}, \texttt{dc1c}, \texttt{msc98-ip}, \texttt{nsr8k}).
Subsequently, only parts of these symmetries are eliminated by a manually enforced presolving step, although all of them seem to be of a similar type as \eqref{eq:toll-like-inequalities}, which allows aggregation.

Also note that this variable aggregation is not possible for $\Symmet{k}$-symmetries in general, in fact not even for $\Symmet{2}$-symmetries.
Like in the instances from Section~\ref{sec:experiments-fibers}, in the following IP the two $\Symmet{2}$-symmetric variables $x$ and $y$ can not be replaced by a single variable because they appear with distinct non-zero coefficients in more than one inequality.
\begin{equation*}
\begin{split}
 \max\quad &x+y\\
 \text{s.t.}\quad & x + y \leq 5\\
& 5x + y \leq 16\\
& x + 5y \leq 16,\qquad\qquad x,y \in \zz
\end{split}
\end{equation*}
So far we have not encountered these kind of symmetries ``in the wild'', for instance, in MIPLIB.

\section{Discussion and open problems}\label{sec:discussion}

In this article we introduced the notion of core sets for convex integer optimization.
The integer feasibility test of a symmetric convex set can be reduced to its core set.
We discussed some structural properties of core sets.
In particular, we obtained a full description of core sets of direct products of symmetric groups.

Based on the core set idea, we proposed two different, straightforward approaches to solve integer linear programs.
First, we considered an algorithm using lattice point enumeration and core set based fiber feasibility tests to solve ILPs.
Our second approach used a parametrization of core sets. 
In this we substituted some of the variables by new ones to restrict the search space to the core set of the given problem.
We implemented prototypes of the proposed algorithms and compared them against commercial solvers, using problems whose symmetry group contains a direct product of symmetric groups.
Our na{\"i}ve first algorithm is faster than Gurobi and CPLEX on generated, highly symmetric problems.
With the second algorithm we were able to solve an open MIPLIB 2010 problem using Gurobi.
Future research is necessary to determine which other problem instances can benefit from these methods.

Many open questions remain:
\begin{itemize}
 \item How can we systematically obtain complete descriptions of core sets for groups other than $\Symmet{n}$ or $\Altern{n}$? 
       In particular, can we characterize groups which have a finite core set?
 \item What can we say about the core set of group compositions other than direct products, e.g. subdirect products and diagonal subgroups, which play an important role in combinatorial optimization problems? 
 \item Can our na{\"i}ve algorithm be improved, for instance by using additional polyhedral information and avoiding the unfavorable lattice point enumeration?
 \item Under what conditions is a core set parametrization based approach advantageous for symmetric problems?
 \item For what kind of symmetric non-linear integer optimization problems are core set based integer feasibility checks useful?
\end{itemize}

\section*{Acknowledgements}

The authors would like to thank the Fields Institute for its hospitality, where larger
parts of this research have been done during the Thematic Program on Discrete Geometry and
Applications.  Further, the authors would like to thank Mathieu Dutour Sikiri\'{c} and
Marc Pfetsch for valuable discussions and the anonymous referee for helpful remarks about
the computational results.  Research by Herr is supported by Studienstiftung des deutschen
Volkes.

\providecommand{\bysame}{\leavevmode\hbox to3em{\hrulefill}\thinspace}
\providecommand{\MR}{\relax\ifhmode\unskip\space\fi MR }
\providecommand{\MRhref}[2]{%
  \href{http://www.ams.org/mathscinet-getitem?mr=#1}{#2}
}
\providecommand{\href}[2]{#2}

\renewcommand{\refname}{Software}


\begin{thebibliography}{19}

\bibitem{bgsv-2012}
C.~Bachoc, D.C.~Gijswijt, A.~Schrijver, F.~Vallentin.
\newblock Invariant semidefinite programs.
\newblock In {\em Handbook of Semidefinite, Conic and Polynomial Optimization}. Springer, 2012.

\bibitem{bhj-2011}
R.~B\"odi, K.~Herr, and M.~Joswig.
\newblock Algorithms for highly symmetric linear and integer programs.
\newblock {\em Math. Program., Ser.A}, pages 65--90, 2013.
\newblock 10.1007/s10107-011-0487-6.

\bibitem{bm-2008}
D.A. Bulutoglu and F.~Margot.
\newblock Classification of orthogonal arrays by integer programming.
\newblock {\em Journal of Statistical Planning and Inference}, 138:654--666,
  2008.

\bibitem{MR543191}
P.J. Davis.
\newblock {\em {Circulant Matrices}}.
\newblock John Wiley \& Sons, New York-Chichester-Brisbane, 1979.
\newblock A Wiley-Interscience Publication, Pure and Applied Mathematics.

\bibitem{friedman-2007}
E.J. Friedman.
\newblock Fundamental domains for integer programs with symmetries.
\newblock In {\em Combinatorial optimization and applications}, volume 4616 of
  {\em Lecture Notes in Comput. Sci.}, pages 146--153. Springer, Berlin, 2007.

\bibitem{gp-2004}
K.~Gatermann and P.A. Parrilo.
\newblock Symmetry groups, semidefinite programs, and sums of squares.
\newblock {\em J. Pure Appl. Algebra}, 192(1-3):95--128, 2004.

\bibitem{MR0103215}
M.~Hall, Jr.
\newblock {\em {The Theory of Groups}}.
\newblock The Macmillan Co., New York, N.Y., 1959.

\bibitem{kp-2008}
V.~Kaibel and M.E. Pfetsch.
\newblock Packing and partitioning orbitopes.
\newblock {\em Math. Program., Ser.A}, 114:1--36, 2008.

\bibitem{miplib}
T.~Koch, T.~Achterberg, E.~Andersen, O.~Bastert, T.~Berthold, R.~E.~Bixby, E.~Danna, G.~Gamrath, A.~M.~Gleixner, S.~Heinz,
A.~Lodi, H.~Mittelmann, T.~Ralphs, D.~Salvagnin, D.~E.~Steffy and K.~Wolter.
\newblock MIPLIB 2010.
\newblock {\em Mathematical Programming Computation}, 3(2):103--163, 2011.

\bibitem{liberti-2012}
L.~Liberti.
\newblock Symmetry in Mathematical Programming.
\newblock In {\em Mixed Integer Nonlinear Programming}, volume 154 of {\em IMA Series}, pages 263--286, Springer, New York, 2012.

\bibitem{lmt-2009}
J.~Linderoth, F.~Margot, and G.~Thain.
\newblock Improving bounds on the football pool problem via symmetry reduction
  and high-throughput computing.
\newblock {\em INFORMS Journal on Computing}, 21:445--457, 2009.

\bibitem{margot-2003}
F.~Margot.
\newblock Exploiting orbits in symmetric {ILP}.
\newblock {\em Mathematical Programming Ser. B}, 98:3--21, 2003.

\bibitem{margot-2009}
F.~Margot.
\newblock Symmetry in integer linear programming.
\newblock In {\em 50 years of integer programming}. Springer, 2009.

\bibitem{MR1239510}
S.~Onn.
\newblock Geometry, complexity, and combinatorics of permutation polytopes.
\newblock {\em J. Combin. Theory Ser. A}, 64(1):31--49, 1993.

\bibitem{olrs-2008}
J.~Ostrowski, J.~Linderoth, F.~Rossi, and S.~Smriglio.
\newblock Orbital branching.
\newblock {\em Mathematical Programming Ser. A}, 126:147--178, 2011.

\bibitem{PfetschRehn2012}
M.~E. Pfetsch, and T.~Rehn.
\newblock Symmetry handling in integer programs revisited.
\newblock In preparation.

\bibitem{MR0095452}
J.~E. Reeve.
\newblock On the volume of lattice polyhedra.
\newblock {\em Proc. london Math. Soc. (3)}, 7:378--395, 1957.

\bibitem{schrijver-1986}
A.~Schrijver.
\newblock {\em Theory of Linear and Integer Programming}.
\newblock Wiley, Chichester, 1986.

\bibitem{wolsey-1998}
L.A. Wolsey.
\newblock {\em Integer programming}.
\newblock Wiley-Interscience Series in Discrete Mathematics and Optimization.
  John Wiley \& Sons Inc., New York, 1998.
\newblock A Wiley-Interscience Publication.

\end{thebibliography}

\begin{thebibliography}{26}
\setcounter{enumiv}{19}

\bibitem{cplex}
ILOG CPLEX.
\newblock \url{http://www.ilog.com/products/cplex/}.

\bibitem{gurobi}
Gurobi.
\newblock \url{http://www.gurobi.com}.

\bibitem{latte}
LattE integrale by {J}.~{De Loera, M.~K\"{o}ppe et al.}
\newblock \url{http://www.math.ucdavis.edu/~latte/}.

\bibitem{normaliz}
{Normaliz} by {W}.~{B}runs, {B}.~{I}chim, and {C}.~{S}\"oger.
\newblock \url{http://www.mathematik.uni-osnabrueck.de/normaliz/}.

\bibitem{PermLib}
{PermLib} by {T}.~{R}ehn.
\newblock \url{http://www.geometrie.uni-rostock.de/software/}.

\bibitem{Polymake}
{Polymake} by E.~Gawrilow, M.~Joswig et al.
\newblock \url{http://polymake.org}.

\bibitem{scip}
SCIP: Solving constraint integer programs.
\newblock \url{http://scip.zib.de/}.

\end{thebibliography}
\end{document}